\documentclass[a4paper]{article}
\usepackage[utf8]{inputenc}
\usepackage{amsmath}
\usepackage{amssymb}
\usepackage{amsthm}
\usepackage{epsfig}
\usepackage{epstopdf}
\usepackage{mathtools}
\usepackage{relsize}
\usepackage{hyperref}
\hypersetup{
    colorlinks=true,     
    linkcolor=blue,     
    citecolor=blue,     
}
\usepackage[numbers,sort&compress]{natbib}
\usepackage{enumitem}
\usepackage{algorithm,algpseudocode}
\usepackage{fullpage}
\allowdisplaybreaks
\usepackage{setspace}
\usepackage{caption}
\usepackage{xspace}
\usepackage{authblk}
\usepackage{subcaption}
\usepackage[export]{adjustbox}

\title{Heteroscedasticity-aware residuals-based \\ contextual stochastic optimization}

\date{January 8, 2021}

\author[1]{Rohit Kannan}
\author[2]{G{\"u}zin Bayraksan}
\author[3]{James R. Luedtke}
\affil[1]{Wisconsin Institute for Discovery, University of Wisconsin-Madison, Madison, WI, USA. \protect\\ E-mail: rohitk@alum.mit.edu}
\affil[2]{Department of Integrated Systems Engineering, The Ohio State University, Columbus, OH, USA. \protect\\ E-mail: bayraksan.1@osu.edu}
\affil[3]{Department of Industrial \& Systems Engineering and Wisconsin Institute for Discovery, \protect\\ University of Wisconsin-Madison, Madison, WI, USA. E-mail: jim.luedtke@wisc.edu}

\newcommand{\tr}[1]{\ensuremath{{#1}^\text{T}}}

\newcommand{\uset}[2]{\ensuremath{\underset{#1}{#2}}}

\DeclarePairedDelimiter\abs{\lvert}{\rvert}%
\DeclarePairedDelimiter\norm{\lVert}{\rVert}%

\newcommand{\prob}[1]{\mathbb{P}\left\lbrace{#1}\right\rbrace}
\newcommand{\expect}[1]{\mathbb{E}\left[{#1}\right]}
\newcommand{\expectation}[2]{\mathbb{E}_{#1}\left[{#2}\right]}

\newcommand{\dev}[2]{\mathbb{D}\left({#1},{#2}\right)}

\newcommand{\convinprob}{\xrightarrow{p}}

\newcommand{\proj}[2]{\operatorname{proj}_{#1}(#2)}

\newcommand{\D}{\mathcal{D}}
\newcommand{\F}{\mathcal{F}}
\newcommand{\Q}{\mathcal{Q}}
\newcommand{\Linf}{L^{\infty}}

\newcommand{\R}{\mathbb{R}}

\newcommand{\hS}{\hat{S}}
\newcommand{\X}{\mathcal{X}}
\newcommand{\Y}{\mathcal{Y}}
\newcommand{\Z}{\mathcal{Z}}

\newcommand{\hf}{\hat{f}}
\newcommand{\hg}{\hat{g}}
\newcommand{\hQ}{\hat{Q}}

\newcommand{\heps}{\hat{\varepsilon}}
\newcommand{\teps}{\tilde{\varepsilon}}

\newcommand{\hth}{\hat{\theta}}
\newcommand{\sth}{\theta^*}
\newcommand{\hv}{\hat{v}}

\newcommand{\hz}{\hat{z}}

\newcommand{\hP}{\hat{\mathcal{P}}}

\newtheorem{theorem}{Theorem}[]
\newtheorem{lemma}[theorem]{Lemma}
\newtheorem{assumption}{Assumption}[]
\newtheorem{remark}{Remark}
\newtheorem{definition}{Definition}
\newtheorem{corollary}[theorem]{Corollary}
\newtheorem{proposition}[theorem]{Proposition}
\newtheorem{example}{Example}
\newtheorem{conjecture}[theorem]{Conjecture}

\let\oldtheorem\theorem
\renewcommand{\theorem}{\oldtheorem\normalfont}

\let\oldlemma\lemma
\renewcommand{\lemma}{\oldlemma\normalfont}

\let\oldassumption\assumption
\renewcommand{\assumption}{\oldassumption\normalfont}

\let\oldremark\remark
\renewcommand{\remark}{\oldremark\normalfont}

\let\olddefinition\definition
\renewcommand{\definition}{\olddefinition\normalfont}

\let\oldcorollary\corollary
\renewcommand{\corollary}{\oldcorollary\normalfont}

\let\oldproposition\proposition
\renewcommand{\proposition}{\oldproposition\normalfont}

\let\oldexample\example
\renewcommand{\example}{\oldexample\normalfont}

\let\oldconjecture\conjecture
\renewcommand{\conjecture}{\oldconjecture\normalfont}

\theoremstyle{definition}
\newtheorem{assp}{Assumption}

\allowdisplaybreaks

\providecommand{\keywords}[1]{{\small \textbf{Key words:} #1}}

\begin{document}

\maketitle

\begin{abstract}
We explore generalizations of some integrated learning and optimization frameworks for data-driven contextual stochastic optimization that can adapt to heteroscedasticity.
We identify conditions on the stochastic program, data generation process, and the prediction setup under which these generalizations possess asymptotic and finite sample guarantees for a  class of stochastic programs, including two-stage stochastic mixed-integer programs with continuous recourse.
We verify that our assumptions hold for popular parametric and nonparametric regression methods.
\\[0.09in]
\end{abstract}

\section{Introduction}

We study data-driven stochastic programming in the presence of covariate/contextual information and examine heteroscedastic cases.
Specifically, we consider the setting where we have a finite number of observations of the uncertain parameters~$Y$ within an optimization model along with simultaneous observations of random covariates~$X$.
Given a new random observation $X = x$, our goal is to solve the {\it conditional stochastic program}
\begin{alignat}{2}
\label{eqn:sp}
&\uset{z \in \Z}{\min} \: && \expect{c(z,Y) \mid X = x}. \tag{SP}
\end{alignat}
Here, $z$ denotes the decision vector, $\Z \subseteq \R^{d_z}$ is the feasible region, and $c: \R^{d_z} \times \R^{d_y} \to \overline{\R}$ is an extended real-valued function.
An example application of this framework is production planning under demand uncertainty~\cite{bertsimas2014predictive}, where products' demands ($Y$) can be predicted using covariates ($X$) such as historical demands, location, and web chatter before making decisions ($z$) on production and inventory levels.
Another application is grid scheduling under wind uncertainty~\cite{donti2017task}, where covariates ($X$) such as weather observations, seasonality, and location can be used to predict available wind power ($Y$) before creating generator schedules ($z$).
Heteroscedasticity arises, for instance, when the variability of product demands or wind power availability depends significantly on the location, seasonality, or other covariates.

\citet{kannan2020data,kannan2020residuals} consider data-driven approaches that integrate a machine learning prediction model within a sample average approximation (SAA) or distributionally robust optimization (DRO) setup to approximate the solution to the conditional stochastic program~\eqref{eqn:sp}; see also~\cite{ban2018dynamic,sen2018learning}.
They first fit a statistical/machine learning model to predict~$Y$ given~$X$ and use this model and its residuals to construct scenarios for $Y$ given $X = x$.
Then, they use these scenarios within an SAA or DRO framework to approximate the solution to~\eqref{eqn:sp}.
We refer the readers to~\cite[e.g.,][]{kannan2020data,kannan2020residuals,ban2018dynamic,sen2018learning,bertsimas2014predictive} for a review of other data-driven approximations to~\eqref{eqn:sp}.

The data-driven formulations in~\citet{kannan2020data,kannan2020residuals} assume that the dependence of the random vector~$Y$ on the random covariates~$X$ can be modeled as $Y = f^*(X) + \varepsilon$, where $f^*(x) := \expect{Y \mid X = x}$ is the regression function and~$\varepsilon$ are zero-mean errors.
These approaches crucially require the errors~$\varepsilon$ to be {\it independent} of the covariates~$X$.
Motivated by applications where such an assumption may fail to hold, we explore generalizations of these approaches that do not require this independence assumption.

\paragraph{Notation.}
Let $[n] := \{1,\dots,n\}$, $\norm{\cdot}$ denote the Euclidean or operator $\ell_2$-norm, $\proj{\mathcal{S}}{v}$ denote the orthogonal projection of $v$ onto a nonempty closed convex set~$\mathcal{S}$, $I$ denote an identity matrix of appropriate dimension, $\tr{v}$ denote the transpose of a vector $v$, and $A \succ 0$ denote that the matrix $A$ is positive definite. Let $\delta$ denote the Dirac measure.
For scalars $c_1, \dots, c_l$, we write $\text{diag}(c_1,\dots,c_l)$ to denote the $l \times l$ diagonal matrix with $i$th diagonal entry equal to $c_i$.
For sets $\mathcal{A}, \mathcal{B} \subseteq \R^{d_z}$, let $\dev{\mathcal{A}}{\mathcal{B}} := \sup_{v \in \mathcal{A}} \text{dist}(v,\mathcal{B})$ denote the deviation of~$\mathcal{A}$ from~$\mathcal{B}$, where $\text{dist}(v,\mathcal{B}) := \inf_{w \in \mathcal{B}} \norm{v - w}$.
The abbreviations `a.e.', `a.s.', `LLN', `i.i.d.', and `r.h.s.' are shorthand for `almost everywhere', `almost surely', `law of large numbers', `independent and identically distributed', and `right-hand side'. 
For a random vector~$V$ with probability measure~$P_V$, we write a.e.\ $v \in V$ to denote $P_V$-a.e.\ $v \in V$.
The symbols~$\xrightarrow{p}$ and $\xrightarrow{a.s.}$ denote convergence in probability and almost surely with respect to the probability measure generating the joint data on~$(Y,X)$.
For random sequences~$\{V_n\}$ and~$\{W_n\}$, we write $V_n = o_p(W_n)$ and $V_n = O_p(W_n)$ to  convey that~$V_n = R_n W_n$ with $\{R_n\}$ converging in probability to zero, or being bounded in probability, respectively.
We write $O(1)$ for generic constants.

\section{Heteroscedasticity-aware residuals-based approximations}

\subsection{Framework and approximations}

To handle heteroscedasticity, we assume that the random vector~$Y$ is related to the random covariates~$X$ as\footnote{We focus our attention on this popular model of heteroscedasticity even though our framework applies more generally, e.g., to relationships of the form $Y = m^*(X,\varepsilon)$ with the mapping $m^*(x,\cdot)$ being {\it invertible} for a.e.\ $x \in \X$ and satisfying some regularity conditions.} $Y = f^*(X) + Q^*(X)\varepsilon$, where $f^*$ denotes the regression function, $Q^*(X)$ is the square root of the conditional covariance matrix of the error term, and the zero-mean random errors~$\varepsilon$ are independent of the covariates~$X$.
This type of model is common in statistics; see, e.g.,~\cite{carroll1982robust,bauwens2006multivariate,zhou2018new,dalalyan2013learning}.
The functions $f^*$ and $Q^*$ are assumed to belong to known classes of functions $\F$ and $\Q$, respectively (which may be infinite-dimensional and depend on the sample size $n$).
Let $\Y \subseteq \R^{d_y}$, $\X \subseteq \R^{d_x}$, and $\Xi \subseteq \R^{d_y}$ denote the supports of $Y$, $X$, and $\varepsilon$, respectively.
Additionally, let $P_{Y \mid X = x}$ denote the conditional distribution of $Y$ given $X = x$ and $P_X$ and $P_{\varepsilon}$ denote the distributions of~$X$ and~$\varepsilon$, respectively.
We assume that~$\Y$ is nonempty and convex and $Q^*(x) \succ 0$ for a.e.\ $x \in \X$.

Under the above assumptions, the conditional stochastic program~\eqref{eqn:sp} is equivalent to
\begin{align}
\label{eqn:speq}
v^*(x) &:= \uset{z \in \Z}{\min} \left\lbrace  g(z;x) := \mathbb{E}\bigl[c(z,f^*(x) + Q^*(x)\varepsilon)\bigr] \right\rbrace,
\end{align}
where the expectation above is computed with respect to the distribution $P_{\varepsilon}$ of~$\varepsilon$.
We assume that the feasible set $\Z \subset \R^{d_z}$ is nonempty and compact, $\expect{\abs{c(z,f^*(x)+\varepsilon)}} < +\infty$ for each $z \in \Z$ and a.e.\ $x \in \X$, and the function~$g(\cdot;x)$ is lower semicontinuous on $\Z$ for a.e.\ $x \in \X$.
These assumptions ensure that problem~\eqref{eqn:speq} is well-defined and its set of optimal solutions $S^*(x)$ is nonempty for a.e.\ $x \in \X$.

Let~$\D_n := \{(y^i,x^i)\}_{i=1}^{n}$ denote the joint observations of $(Y,X)$ and $\{\varepsilon^i\}_{i=1}^{n}$ denote the corresponding realizations of the errors~$\varepsilon$. Note that these realizations of~$\varepsilon$ satisfy
\[
\varepsilon^i = \bigl[Q^*(x^i)\bigr]^{-1} (y^i - f^*(x^i)), \quad \forall i \in [n].
\]
If we know the functions~$f^*$ and $Q^*$, then we can construct the following \textit{full-information SAA} (FI-SAA) to problem~\eqref{eqn:speq} using the data~$\D_n$:
\begin{align}
\label{eqn:fullinfsaa}
&\uset{z \in \Z}{\min} \biggl\{ g^*_n(z;x) := \dfrac{1}{n} \displaystyle\sum_{i=1}^{n} c(z,f^*(x) + Q^*(x)\varepsilon^i) \biggr\}.
\end{align}

Because the functions~$f^*$ and~$Q^*$ are unknown, we first estimate them by $\hf_n$ and $\hQ_n$, respectively, using a regression method on the data~$\D_n$ (see Section~\ref{sec:regr} for details).
Assuming that the estimate~$\hQ_n$ is a.s.\ positive definite on~$\X$ (i.e., it a.s.\ satisfies $\hQ_n(x) \succ 0$ for a.e.\ $x \in \X$), we then use the empirical estimates 
\[
\heps^i_n := \bigl[\hQ_n(x^i)\bigr]^{-1} (y^i - \hf_n(x^i)), \quad \forall i \in [n],
\]
of $\{\varepsilon^i\}_{i=1}^{n}$ to construct the following {\it empirical residuals-based SAA} (ER-SAA) to problem~\eqref{eqn:speq} in the heteroscedastic setting (cf.\ \cite{kannan2020data,kannan2020residuals})\footnote{We can also construct similar generalizations of the Jackknife-based SAAs in~\cite{kannan2020data}.}:
\begin{align}
\label{eqn:app}
\hv^{ER}_n(x) &:= \uset{z \in \Z}{\min} \biggl\{ \hg^{ER}_n(z;x) := \dfrac{1}{n}\displaystyle\sum_{i=1}^{n} c\bigl(z,\proj{\Y}{\hf_n(x) + \hQ_n(x)\heps^i_{n}}\bigr)\biggr\}.
\end{align}
Let $\hz^{ER}_n(x)$ denote an optimal solution to problem~\eqref{eqn:app} and~$\hS^{ER}_n(x)$ denote its optimal solution set. 
Additionally, let $P^*_n(x)$ and $\hat{P}^{ER}_n(x)$ denote the estimates of the conditional distribution $P_{Y \mid X = x}$ of $Y$ given $X = x$ corresponding to the FI-SAA problem~\eqref{eqn:fullinfsaa} and ER-SAA problem~\eqref{eqn:app}, respectively, i.e.,  
\[
P^*_n(x) := \frac{1}{n} \sum_{i=1}^{n} \delta_{f^*(x)+Q^*(x)\varepsilon^i} \quad \text{and} \quad \hat{P}^{ER}_n(x) := \dfrac{1}{n} \sum_{i=1}^{n} \delta_{\proj{\Y}{\hf_n(x) + \hQ_n(x)\heps^i_{n}}}.
\]
When we only have a limited number of observations $n$, the following residuals-based DRO formulation provides an alternative to the ER-SAA problem~\eqref{eqn:app} that can yield solutions with better out-of-sample performance (cf.\ \cite{kannan2020residuals}):
\begin{alignat}{2}
\label{eqn:dro}
&\uset{z \in \Z}{\min} \: \uset{Q \in \hP_n(x)}{\sup} \: && \expectation{Y \sim Q}{c(z,Y)},
\end{alignat}
where $\hP_n(x)$ is an ambiguity set for~$P_{Y \mid X = x}$.
Following~\cite{kannan2020residuals}, we call problem~\eqref{eqn:dro} with~$\hP_n(x)$ centered at $\hat{P}^{ER}_n(x)$ the {\it empirical residuals-based DRO} (ER-DRO) problem.

\subsection{Theoretical results for the heteroscedastic setting}

For the homoscedastic case, i.e., when $Q^* \equiv \hQ_n \equiv I$ and so the model class~$\Q$ comprises only the constant function $Q: x \mapsto I$, $\forall x \in \X$,~\citet{kannan2020data,kannan2020residuals} investigate conditions under which the optimal value of problems~\eqref{eqn:app} and~\eqref{eqn:dro} asymptotically converge in probability to those of the true problem~\eqref{eqn:speq}.
They also identify conditions under which every accumulation point of a sequence of optimal solutions to problems~\eqref{eqn:app} and~\eqref{eqn:dro} is in probability an optimal solution to problem~\eqref{eqn:speq} and outline conditions under which solutions to problems~\eqref{eqn:app} and~\eqref{eqn:dro} possess finite sample guarantees.
An integral part of this analysis is bounding a distance between the empirical distributions $\hat{P}^{ER}_n(x)$ and $P^*_n(x)$. 

By the Lipschitz continuity of orthogonal projections, we have for each $x \in \X$
\begin{align*}
&\qquad\qquad \norm{\proj{\Y}{\hf_n(x) + \hQ_n(x)\heps^i_{n}} - (f^*(x) + Q^*(x)\varepsilon^i)} \leq \norm{\teps^{i}_{n}(x)}, \qquad \forall i \in [n], \end{align*}
where the $i$th deviation term $\teps^{i}_{n}(x)$ is given by
\begin{align*}
&\teps^{i}_{n}(x) := (\hf_n(x) + \hQ_n(x)\heps^i_{n}) - (f^*(x) + Q^*(x)\varepsilon^i).
\end{align*}
The analysis in~\cite{kannan2020data,kannan2020residuals} implies that under certain assumptions on the stochastic program~\eqref{eqn:speq}, asymptotic and finite sample guarantees on the power mean deviation term $(\frac{1}{n} \sum_{i=1}^{n} \norm{\teps^i_n(x)}^p)^{1/p}$ for a suitable value of $p \geq 1$ translate to asymptotic and finite sample guarantees on the optimal value and optimal solutions to the ER-SAA problem~\eqref{eqn:app} and the ER-DRO problem~\eqref{eqn:dro}.
Specifically, the analyses in~\cite{kannan2020data,kannan2020residuals} imply that theoretical guarantees on the term $(\frac{1}{n} \sum_{i=1}^{n} \norm{\teps^i_n(x)}^p)^{1/p}$ for $p = 1$ and $p = 2$ translate to theoretical guarantees on solutions to problems~\eqref{eqn:app} and~\eqref{eqn:dro} for a class of two-stage stochastic mixed-integer programs (MIPs) with continuous recourse and, in the ER-DRO setting, to broad families of ambiguity sets.

We now provide concrete examples of how guarantees on the {\it mean deviation term} $\frac{1}{n} \sum_{i=1}^{n} \norm{\teps^i_n(x)}$ (i.e., when $p = 1$) translate to guarantees on the ER-SAA problem~\eqref{eqn:app}. 
In addition to focusing on the ER-SAA problem~\eqref{eqn:app} for brevity, we narrow our attention to stochastic programs~\eqref{eqn:speq} whose objective function satisfies the following Lipschitz condition.

\begin{assumption}
\label{ass:equilipschitz}
For each $z \in \Z$, the function~$c(z,\cdot)$ is Lipschitz continuous on~$\Y$ with Lipschitz constant $L(z)$ satisfying $\sup_{z \in \Z} L(z) < +\infty$.
\end{assumption}

As an example, Appendix~EC.2 of~\cite{kannan2020data} verifies that Assumption~\ref{ass:equilipschitz} holds for two-stage stochastic MIPs with continuous recourse under mild conditions.
For extensions of the below results to a broader class of stochastic programs~\eqref{eqn:speq} and to the ER-DRO problem~\eqref{eqn:dro}, 
we refer the readers to~\cite{kannan2020data,kannan2020residuals}.

We now list conditions on the FI-SAA problem~\eqref{eqn:fullinfsaa} under which consistency and asymptotic optimality, rates of convergence, and finite sample guarantees---to be defined precisely in respective theorems below---can be achieved for the ER-SAA approximation~\eqref{eqn:app} of the true problem~\eqref{eqn:speq} in the heteroscedastic setting.
As mentioned, a key component of this analysis requires respective conditions to be satisfied by the mean deviation term; these are investigated in Section~\ref{sec:bounds}.
Section~\ref{sec:regr} presents examples of regression/learning setups that satisfy the assumptions set forth for the heteroscedastic setting.

We begin with a uniform weak LLN assumption on the FI-SAA objective (see Assumption~3 of~\cite{kannan2020data} and the surrounding discussion for conditions under which it holds).
Along with suitable convergence of the mean deviation term, this assumption helps us establish {\it uniform} convergence in probability of the sequence of objective functions of the ER-SAA problem~\eqref{eqn:app} to the objective function of the true problem~\eqref{eqn:speq} on the feasible set~$\Z$ (see Proposition~1 of~\cite{kannan2020data}). This in turn provides the building block for consistency and asymptotic optimality.

\begin{assumption}
\label{ass:uniflln}
For a.e.~$x \in \X$, the sequence of sample average objective functions $\left\lbrace g^*_n(\cdot;x) \right\rbrace$ of the FI-SAA problem~\eqref{eqn:fullinfsaa} converges in probability to the objective function $g(\cdot;x)$ of the true problem~\eqref{eqn:speq} uniformly on the set~$\Z$.
\end{assumption}

Our first result implies that consistency of the mean deviation term $\frac{1}{n} \sum_{i=1}^{n} \norm{\teps^i_n(x)}$ translates to consistency and asymptotic optimality of solutions to the ER-SAA problem~\eqref{eqn:app}.

\begin{theorem}[{\bf Consistency and asymptotic optimality}]
\label{thm:consist}
Suppose Assumptions~\ref{ass:equilipschitz} and~\ref{ass:uniflln} hold and the mean deviation term converges to zero in probability, i.e., $\frac{1}{n} \sum_{i=1}^{n} \norm{\teps^i_n(x)} \convinprob 0$ for a.e.\ $x \in \X$.
Then for a.e.\ $x \in \X$
\[
\hv^{ER}_n(x) \xrightarrow{p} v^*(x), \quad \mathbb{D}\bigl(\hS^{ER}_n(x),S^*(x)\bigr) \xrightarrow{p} 0, \quad \text{and} \quad \sup_{z \in \hS^{ER}_n(x)} g(z;x) \xrightarrow{p} v^*(x).
\]
\end{theorem}
\begin{proof}
See the proofs of Proposition~1 and Theorem~1 of~\citet{kannan2020data}.
\end{proof}

Next, we refine Assumption~\ref{ass:uniflln} to assume that the sequence of objective functions of the FI-SAA problem~\eqref{eqn:fullinfsaa} converges to the objective function of the true problem~\eqref{eqn:speq} at a suitable rate (see Assumption~5 of~\cite{kannan2020data} and the surrounding discussion for conditions under which it holds).

\begin{assumption}
\label{ass:functionalclt}
The function~$c$ in problem~\eqref{eqn:speq} and the data~$\D_n$ satisfy the following functional central limit theorem for the FI-SAA objective:
\[
\sqrt{n} \left( g^*_n(\cdot;x) - g(\cdot;x) \right) \xrightarrow{d} V(\cdot;x), \quad \text{for a.e. } x \in \X,
\]
where $g^*_n(\cdot;x)$, $g(\cdot;x)$, and $V(\cdot;x)$ are (random) elements of $\Linf(\Z)$, the Banach space of essentially bounded functions on $\Z$ equipped with the supremum norm.
\end{assumption}

Our second result implies that rates of convergence of the mean deviation term $\frac{1}{n} \sum_{i=1}^{n} \norm{\teps^i_n(x)}$ to zero directly translate to rates of convergence of the suboptimality of ER-SAA solutions to zero.

\begin{theorem}[{\bf Rate of convergence}]
\label{thm:convrate}
Suppose Assumptions~\ref{ass:equilipschitz} and~\ref{ass:functionalclt} hold and there exists a constant $r \in (0,1]$ such that $\frac{1}{n} \sum_{i=1}^{n} \norm{\teps^i_n(x)} = O_p(n^{-r/2})$ for a.e.\ $x \in \X$.
Then, for a.e.\ $x \in \X$
\[
\abs*{\hv^{ER}_n(x) - v^*(x)} = O_p(n^{-r/2}) \quad \text{and} \quad \abs*{g(\hz^{ER}_n(x);x) - v^*(x)} = O_p(n^{-r/2}).
\]
\end{theorem}
\begin{proof}
Follows from the proof of Theorem~11 of~\citet{kannan2020residuals} (cf.\ Theorem~2 of~\cite{kannan2020data}).
\end{proof}

Finally, we refine Assumption~\ref{ass:functionalclt} to assume that the sequence of objectives of the FI-SAA problem~\eqref{eqn:fullinfsaa} possess a finite sample guarantee (see~\citep[Assumption~7]{kannan2020data} and the discussion after it for conditions under which it holds).

\begin{assumption}
\label{ass:tradsaalargedev}
The FI-SAA problem~\eqref{eqn:fullinfsaa} possesses the following uniform exponential bound property: for any constant $\kappa > 0$ and a.e.~$x \in \X$, there exist positive constants $K(\kappa,x)$ and $\beta(\kappa,x)$ such that
\[
\mathbb{P} \Bigl\{\uset{z \in \Z}{\sup} \: \abs*{ g^*_n(z;x) - g(z;x)} > \kappa \Bigr\} \leq K(\kappa,x) \exp(-n\beta(\kappa,x)), \quad \forall n \in \mathbb{N}.
\]
\end{assumption}

Our final result of this section implies that finite sample guarantees on the mean deviation term $\frac{1}{n} \sum_{i=1}^{n} \norm{\teps^i_n(x)}$ translate to finite sample guarantees on solutions to the ER-SAA problem~\eqref{eqn:app}.

\begin{theorem}[{\bf Finite sample guarantee}]
\label{thm:finitesample}
Suppose Assumptions~\ref{ass:equilipschitz} and~\ref{ass:tradsaalargedev} hold and for any constant $\kappa > 0$ and a.e.\ $x \in \X$, there exist positive constants $\tilde{K}(\kappa,x)$ and $\tilde{\beta}(\kappa,x)$ such that 
\[
\mathbb{P}\biggl\{ \frac{1}{n} \sum_{i=1}^{n} \norm{\teps^i_n(x)} > \kappa\biggr\} \leq \tilde{K}(\kappa,x) \exp\bigl(-n\tilde{\beta}(\kappa,x)\bigr), \quad \forall n \in \mathbb{N}.
\]
Then, for a.e.\ $x \in \X$, given constant $\eta > 0$, there exist positive constants $Q(\eta,x)$ and $\gamma(\eta,x)$ (depending on $K$, $\tilde{K}$, $\beta$, and $\tilde{\beta}$) such that
\[\prob{\textup{dist}(\hz^{ER}_n(x),S^*(x)) \geq \eta} \leq Q(\eta,x) \exp(-n\gamma(\eta,x)), \quad \forall n \in \mathbb{N}.
\]
\end{theorem}
\begin{proof}
See Theorem~3 of~\citet{kannan2020data}.
\end{proof}

In the remainder of this note, we identify conditions under which the asymptotic and finite sample guarantees required by Theorems~\ref{thm:consist},~\ref{thm:convrate}, and~\ref{thm:finitesample} hold for the mean deviation term $\frac{1}{n} \sum_{i=1}^{n} \norm{\teps^i_n(x)}$.
A similar analysis can be carried out for the root-mean-square deviation term $(\frac{1}{n} \sum_{i=1}^{n} \norm{\teps^i_n(x)}^2)^{1/2}$, which is required by~\cite{kannan2020residuals} for the analysis of phi-divergence-based ER-DRO problems~\eqref{eqn:dro} for stochastic programs satisfying Assumption~\ref{ass:equilipschitz}.
We omit these details for brevity.

\section{Guarantees for the mean deviation term}
\label{sec:bounds}

In this section, we investigate conditions under which the mean deviation term $\frac{1}{n} \sum_{i=1}^{n} \norm{\teps^i_n(x)}$ converges to zero in probability at a certain rate and possesses finite sample guarantees.
We begin by bounding the mean deviation in terms of the functions~$f^*$ and~$Q^*$, their regression estimates~$\hf_n$ and~$\hQ_n$, and the data~$\D_n$.
Throughout, we implicitly assume that the estimate~$\hQ_n$ a.s.\ satisfies $\hQ_n(x) \succ 0$ for a.e.\ $x \in \X$, which can be guaranteed by an appropriate choice of the model class~$\Q$.

\subsection{Bounding the mean deviation term}

We begin by noting that
\begin{align}
\label{eqn:meandeviation_int1}
\frac{1}{n} \sum_{i=1}^{n} \norm{\teps^i_n(x)} &= \frac{1}{n} \sum_{i=1}^{n} \norm{(\hf_n(x) + \hQ_n(x)\heps^i_{n}) - (f^*(x) + Q^*(x)\varepsilon^i)} \nonumber\\
&\leq \norm{\hf_n(x) - f^*(x)} + \frac{1}{n} \sum_{i=1}^{n} \norm{\hQ_n(x)\heps^i_{n} - Q^*(x)\varepsilon^i}.
\end{align}
We now bound the second term on the r.h.s.\ of inequality~\eqref{eqn:meandeviation_int1}.
We have
\begin{align}
\label{eqn:meandeviation_int2}
&\frac{1}{n} \sum_{i=1}^{n} \norm{\hQ_n(x)\heps^i_{n} - Q^*(x)\varepsilon^i} \nonumber\\
=& \frac{1}{n} \sum_{i=1}^{n} \bigl\lVert\hQ_n(x)\bigl[\hQ_n(x^i)\bigr]^{-1} (y^i - \hf_n(x^i)) - Q^*(x)\bigl[Q^*(x^i)\bigr]^{-1} (y^i - f^*(x^i))\bigr\rVert \nonumber\\
=& \frac{1}{n} \sum_{i=1}^{n} \bigl\lVert\hQ_n(x)\bigl[\hQ_n(x^i)\bigr]^{-1} \bigl(y^i - f^*(x^i) + f^*(x^i) - \hf_n(x^i)\bigr) - Q^*(x)\bigl[Q^*(x^i)\bigr]^{-1} (y^i - f^*(x^i))\bigr\rVert \nonumber\\
\leq& \frac{1}{n} \sum_{i=1}^{n} \bigl\lVert \bigl(\hQ_n(x)\bigl[\hQ_n(x^i)\bigr]^{-1} - Q^*(x)\bigl[Q^*(x^i)\bigr]^{-1} \bigr) (y^i - f^*(x^i))\bigr\rVert + \frac{1}{n} \sum_{i=1}^{n} \bigl\lVert \hQ_n(x)\bigl[\hQ_n(x^i)\bigr]^{-1} (f^*(x^i) - \hf_n(x^i))\bigr\rVert \nonumber\\
=& \frac{1}{n} \sum_{i=1}^{n} \bigl\lVert \bigl(\hQ_n(x)\bigl[\hQ_n(x^i)\bigr]^{-1} - Q^*(x)\bigl[Q^*(x^i)\bigr]^{-1}\bigr) Q^*(x^i) \varepsilon^i \bigr\rVert + \frac{1}{n} \sum_{i=1}^{n} \bigl\lVert\hQ_n(x) \bigl[\hQ_n(x^i)\bigr]^{-1} (f^*(x^i) - \hf_n(x^i)) \bigr\rVert,
\end{align}
where the final step follows from the definition of $\{\varepsilon^i\}_{i=1}^{n}$.
We have for each $i \in [n]$
\begin{align*}
\hQ_n(x) \bigl[\hQ_n(x^i)\bigr]^{-1} - Q^*(x)\bigl[Q^*(x^i)\bigr]^{-1} &= \hQ_n(x) \bigl(\bigl[\hQ_n(x^i)\bigr]^{-1} - \bigl[Q^*(x^i)\bigr]^{-1}\bigr) + [\hQ_n(x) - Q^*(x)] \bigl[Q^*(x^i)\bigr]^{-1}.
\end{align*}
Plugging the above equality into inequality~\eqref{eqn:meandeviation_int2}, we get
\begin{align}
\label{eqn:meandeviation_int3}
&\frac{1}{n} \sum_{i=1}^{n} \norm{\hQ_n(x)\heps^i_{n} - Q^*(x)\varepsilon^i} \nonumber\\
\leq& \frac{1}{n} \sum_{i=1}^{n} \Bigl( \norm{\hQ_n(x)} \bigl\lVert \bigl[\hQ_n(x^i)\bigr]^{-1} - \bigl[Q^*(x^i)\bigr]^{-1}\bigr\rVert \norm{Q^*(x^i)} + \norm{\hQ_n(x) - Q^*(x)} \Bigr) \norm{\varepsilon^i} + \nonumber\\
& \quad \frac{1}{n} \sum_{i=1}^{n} \norm{\hQ_n(x)} \bigl\lVert\bigl[\hQ_n(x^i)\bigr]^{-1}\bigr\rVert \norm{f^*(x^i) - \hf_n(x^i)} \nonumber \\
\leq& \norm{\hQ_n(x)} \biggl(\frac{1}{n} \sum_{i=1}^{n} \bigl\lVert \bigl[\hQ_n(x^i)\bigr]^{-1} - \bigl[Q^*(x^i)\bigr]^{-1}\bigr\rVert^2\biggr)^{1/2} \biggl(\frac{1}{n} \sum_{i=1}^{n} \norm{Q^*(x^i)}^4\biggr)^{1/4} \biggl(\frac{1}{n} \sum_{i=1}^{n} \norm{\varepsilon^i}^4\biggr)^{1/4} + \\ 
&\quad \norm{\hQ_n(x) - Q^*(x)}\biggl(\frac{1}{n} \sum_{i=1}^{n} \norm{\varepsilon^i}\biggr) + \norm{\hQ_n(x)} \biggl(\frac{1}{n} \sum_{i=1}^{n} \bigl\lVert \bigl[\hQ_n(x^i)\bigr]^{-1}\bigr\rVert^2\biggr)^{1/2} \biggl(\frac{1}{n} \sum_{i=1}^{n} \norm{f^*(x^i) - \hf_n(x^i)}^2\biggr)^{1/2}, \nonumber
\end{align}
where the last step above follows by repeated application of the Cauchy-Schwarz inequality.
Finally, using inequality~\eqref{eqn:meandeviation_int3} in inequality~\eqref{eqn:meandeviation_int1}, we get
\begin{align}
\label{eqn:meandeviation}
\frac{1}{n} \sum_{i=1}^{n} \norm{\teps^i_n(x)} &\leq \norm{\hf_n(x) - f^*(x)} + \norm{\hQ_n(x) - Q^*(x)}\biggl(\frac{1}{n} \sum_{i=1}^{n} \norm{\varepsilon^i}\biggr) + \nonumber \\
& \quad \norm{\hQ_n(x)} \biggl(\frac{1}{n} \sum_{i=1}^{n} \bigl\lVert \bigl[\hQ_n(x^i)\bigr]^{-1} - \bigl[Q^*(x^i)\bigr]^{-1}\bigr\rVert^2\biggr)^{1/2} \biggl(\frac{1}{n} \sum_{i=1}^{n} \norm{Q^*(x^i)}^4\biggr)^{1/4} \biggl(\frac{1}{n} \sum_{i=1}^{n} \norm{\varepsilon^i}^4\biggr)^{1/4} + \nonumber\\ 
& \quad \norm{\hQ_n(x)} \biggl(\frac{1}{n} \sum_{i=1}^{n} \bigl\lVert \bigl[\hQ_n(x^i)\bigr]^{-1}\bigr\rVert^2\biggr)^{1/2} \biggl(\frac{1}{n} \sum_{i=1}^{n} \norm{f^*(x^i) - \hf_n(x^i)}^2\biggr)^{1/2}.
\end{align}
In the remainder of this section, we rely on inequality~\eqref{eqn:meandeviation} to identify conditions under which the mean deviation term $\frac{1}{n} \sum_{i=1}^{n} \norm{\teps^i_n(x)}$ possesses asymptotic and finite sample guarantees.
We postpone the verification of these assumptions to Section~\ref{sec:regr}.

Before we proceed, we mention alternative ways to bound the mean deviation term $\frac{1}{n} \sum_{i=1}^{n} \norm{\teps^i_n(x)}$ that may be easier to verify in some contexts.
By slightly changing some of the steps leading to inequality~\eqref{eqn:meandeviation_int3}, the third term on the 
r.h.s.\ of inequality~\eqref{eqn:meandeviation} can be replaced with the term
\[
\norm{\hQ_n(x)} \biggl(\frac{1}{n} \sum_{i=1}^{n} \bigl\lVert \bigl[\hQ_n(x^i)\bigr]^{-1} Q^*(x^i) - I\bigr\rVert^2\biggr)^{1/2} \biggl(\frac{1}{n} \sum_{i=1}^{n} \norm{\varepsilon^i}^2\biggr)^{1/2}.
\]
When the second term in the expression above possesses the requisite asymptotic and finite sample guarantees (see, e.g.,~\citep[Section~3]{zhou2018new}),
this yields an alternative form of inequality~\eqref{eqn:meandeviation} that requires milder assumptions on the distribution of the errors~$\varepsilon$.
For another alternative, notice that the first term on the r.h.s.\ of inequality~\eqref{eqn:meandeviation_int2} can also be bounded from above as
\begin{align*}
&\frac{1}{n} \sum_{i=1}^{n} \bigl\lVert \bigl(\hQ_n(x)\bigl[\hQ_n(x^i)\bigr]^{-1} - Q^*(x)\bigl[Q^*(x^i)\bigr]^{-1}\bigr) Q^*(x^i) \varepsilon^i \bigr\rVert \\
=&\frac{1}{n} \sum_{i=1}^{n} \bigl\lVert \bigl(\hQ_n(x)\bigl[\hQ_n(x^i)\bigr]^{-1} Q^*(x^i) - Q^*(x)\bigr) \varepsilon^i \bigr\rVert \\
=&\frac{1}{n} \sum_{i=1}^{n} \bigl\lVert \hQ_n(x)\bigl(\bigl[\hQ_n(x^i)\bigr]^{-1}Q^*(x^i) - I\bigr) \varepsilon^i + (\hQ_n(x) - Q^*(x)) \varepsilon^i \bigr\rVert \\
\leq& \norm{\hQ_n(x)}\Bigl(\uset{\bar{x} \in \X}{\sup} \norm{\bigl[\hQ_n(\bar{x})\bigr]^{-1}}\Bigr) \biggl(\frac{1}{n} \sum_{i=1}^{n} \norm{Q^*(x^i) - \hQ_n(x^i)}^2\biggr)^{1/2} \biggl(\frac{1}{n} \sum_{i=1}^{n} \norm{\varepsilon^i}^2\biggr)^{1/2} + \norm{\hQ_n(x) - Q^*(x)}\biggl(\frac{1}{n} \sum_{i=1}^{n} \norm{\varepsilon^i}\biggr),
\end{align*}
where the final step follows by the Cauchy-Schwarz inequality.
Additionally, the second term on the r.h.s.\ of inequality~\eqref{eqn:meandeviation_int2} can also be bounded from above as
\begin{align*}
\frac{1}{n} \sum_{i=1}^{n} \bigl\lVert\hQ_n(x) \bigl[\hQ_n(x^i)\bigr]^{-1} (f^*(x^i) - \hf_n(x^i)) \bigr\rVert &\leq \norm{\hQ_n(x)} \Bigl(\uset{\bar{x} \in \X}{\sup} \norm{\bigl[\hQ_n(\bar{x})\bigr]^{-1}}\Bigr) \biggl(\frac{1}{n} \sum_{i=1}^{n} \norm{f^*(x^i) - \hf_n(x^i)}\biggr).
\end{align*}
Using these bounds in inequality~\eqref{eqn:meandeviation_int2}, we conclude that the mean deviation $\frac{1}{n} \sum_{i=1}^{n} \norm{\teps^i_n(x)}$ can also be bounded from above as
\begin{align}
\label{eqn:meandeviation_alt}
\frac{1}{n} \sum_{i=1}^{n} \norm{\teps^i_n(x)} &\leq \norm{\hf_n(x) - f^*(x)} + \norm{\hQ_n(x) - Q^*(x)}\biggl(\frac{1}{n} \sum_{i=1}^{n} \norm{\varepsilon^i}\biggr) + \nonumber \\
& \quad \norm{\hQ_n(x)}\Bigl(\uset{\bar{x} \in \X}{\sup} \norm{\bigl[\hQ_n(\bar{x})\bigr]^{-1}}\Bigr) \biggl(\frac{1}{n} \sum_{i=1}^{n} \norm{Q^*(x^i) - \hQ_n(x^i)}^2\biggr)^{1/2} \biggl(\frac{1}{n} \sum_{i=1}^{n} \norm{\varepsilon^i}^2\biggr)^{1/2} + \nonumber\\ 
& \quad \norm{\hQ_n(x)} \Bigl(\uset{\bar{x} \in \X}{\sup} \norm{\bigl[\hQ_n(\bar{x})\bigr]^{-1}}\Bigr) \biggl(\frac{1}{n} \sum_{i=1}^{n} \norm{f^*(x^i) - \hf_n(x^i)}\biggr).
\end{align}
Inequality~\eqref{eqn:meandeviation_alt} can be used to derive alternative conditions under which our asymptotic and finite sample guarantees hold.
For instance, asymptotic and finite sample guarantees on the uniform convergence of the estimate~$\hQ_n$ to~$Q^*$ on~$\X$ directly translate to the requisite asymptotic and finite sample guarantees on the quantities involving the estimate~$\hQ_n$ in~\eqref{eqn:meandeviation_alt}.
These conditions again necessitate milder assumptions on the distribution of the errors~$\varepsilon$ relative to~\eqref{eqn:meandeviation}; however, they require the function~$Q^*$ and its regression estimate~$\hQ_n$ to be (asymptotically) a.s.\ uniformly invertible (cf.\ \cite{robinson1987asymptotically}), i.e., $\sup_{\bar{x} \in \X} \norm{[Q^*(\bar{x})]^{-1}} < +\infty$ and a.s.\ (for $n$ large enough) $\sup_{\bar{x} \in \X} \norm{\bigl[\hQ_n(\bar{x})\bigr]^{-1}} < +\infty$.
We omit these details for brevity and continue with inequality~\eqref{eqn:meandeviation} for the rest of our analysis.

\subsection{Consistency}

We begin with assumptions that guarantee that the mean deviation term $\frac{1}{n} \sum_{i=1}^{n} \norm{\teps^i_n(x)}$ converges to zero in probability.

\begin{assumption}
\label{ass:varweaklln}
The function~$Q^*$ and the data~$\D_n$ satisfy the weak LLNs 
\[
\frac{1}{n} \sum_{i=1}^{n} \norm{Q^*(x^i)}^4 \convinprob \mathbb{E}[\norm{Q^*(X)}^4] \quad \text{and} \quad \frac{1}{n} \sum_{i=1}^{n} \bigl\lVert \bigl[Q^*(x^i)\bigr]^{-1}\bigr\rVert^2 \convinprob \mathbb{E}\bigl[ \bigl\lVert \bigl[Q^*(X)\bigr]^{-1} \bigr\rVert^2 \bigr].
\]
\end{assumption}

\begin{assumption}
\label{ass:errorsweaklln}
The samples~$\{\varepsilon^i\}_{i=1}^{n}$ satisfy the weak LLN $\frac{1}{n} \sum_{i=1}^{n} \norm{\varepsilon^i}^4 \convinprob \mathbb{E}[\norm{\varepsilon}^4]$.
\end{assumption}

Assumptions~\ref{ass:varweaklln} and~\ref{ass:errorsweaklln} are mild weak LLN assumptions that hold, for instance, when the samples $\{(x^i,\varepsilon^i)\}$ are i.i.d.\ and the quantities $\mathbb{E}[\norm{Q^*(X)}^4]$, $\mathbb{E}\bigl[ \bigl\lVert \bigl[Q^*(X)\bigr]^{-1} \bigr\rVert^2 \bigr]$, and $\mathbb{E}[\norm{\varepsilon}^4]$ are finite.
They also hold for non-i.i.d.\ data arising from mixing/stationary processes that satisfy suitable assumptions (see the discussion following Assumption~3 of~\cite{kannan2020data}).
We also require the following consistency assumption on the regression estimates $\hf_n$ and $\hQ_n$ (cf.\ Assumption~4 of~\cite{kannan2020data}).

\begin{assumption}
\label{ass:regconsist}
The regression estimates $\hf_n$ and $\hQ_n$ possess the following consistency properties:
\begin{align*}
&\hf_n(x) \convinprob f^*(x) \quad \text{and} \quad \hQ_n(x) \convinprob Q^*(x), \quad \text{for a.e. } x \in \X, \quad \text{and} \\
&\frac{1}{n} \sum_{i=1}^{n} \norm{\hf_n(x^i) - f^*(x^i)}^2 \convinprob 0, \quad \frac{1}{n} \sum_{i=1}^{n} \bigl\lVert \bigl[\hQ_n(x^i)\bigr]^{-1} - \bigl[Q^*(x^i)\bigr]^{-1}\bigr\rVert^2 \convinprob 0.
\end{align*}
\end{assumption}

The following result will prove useful in our analysis.

\begin{lemma}
\label{lem:varbound}
We have
\[
\biggl(\frac{1}{n} \sum_{i=1}^{n} \bigl\lVert \bigl[\hQ_n(x^i)\bigr]^{-1}\bigr\rVert^2\biggr)^{1/2} \leq \biggl(\frac{1}{n} \sum_{i=1}^{n} \bigl\lVert \bigl[\hQ_n(x^i)\bigr]^{-1} - \bigl[Q^*(x^i)\bigr]^{-1}\bigr\rVert^2\biggr)^{1/2} + \biggl(\frac{1}{n} \sum_{i=1}^{n} \bigl\lVert \bigl[Q^*(x^i)\bigr]^{-1}\bigr\rVert^2\biggr)^{1/2}.
\]
\end{lemma}
\begin{proof}
The triangle inequality for the operator norm implies 
\[
\bigl\lVert \bigl[\hQ_n(x^i)\bigr]^{-1}\bigr\rVert \leq \bigl\lVert \bigl[\hQ_n(x^i)\bigr]^{-1} - \bigl[Q^*(x^i)\bigr]^{-1}\bigr\rVert + \bigl\lVert \bigl[Q^*(x^i)\bigr]^{-1}\bigr\rVert, \quad \forall i \in [n].
\]
Therefore, the following component-wise inequality holds:
\[
0 \leq \begin{pmatrix} \bigl\lVert \bigl[\hQ_n(x^1)\bigr]^{-1}\bigr\rVert \\ \vdots \\ \bigl\lVert \bigl[\hQ_n(x^n)\bigr]^{-1}\bigr\rVert \end{pmatrix} \leq \begin{pmatrix} \bigl\lVert \bigl[\hQ_n(x^1)\bigr]^{-1} - \bigl[Q^*(x^1)\bigr]^{-1}\bigr\rVert \\ \vdots \\ \bigl\lVert \bigl[\hQ_n(x^n)\bigr]^{-1} - \bigl[Q^*(x^n)\bigr]^{-1}\bigr\rVert \end{pmatrix} + \begin{pmatrix} \bigl\lVert \bigl[Q^*(x^1)\bigr]^{-1}\bigr\rVert \\ \vdots \\ \bigl\lVert \bigl[Q^*(x^n)\bigr]^{-1}\bigr\rVert \end{pmatrix}.
\]
The stated result then follows as a consequence of the triangle inequality for the $\ell_2$-norm.
\end{proof}

Applying Assumptions~\ref{ass:varweaklln},~\ref{ass:errorsweaklln}, and~\ref{ass:regconsist} to inequality~\eqref{eqn:meandeviation} immediately yields the following result.

\begin{theorem}
Suppose Assumptions~\ref{ass:varweaklln},~\ref{ass:errorsweaklln}, and~\ref{ass:regconsist} hold.
Then $\frac{1}{n} \sum_{i=1}^{n} \norm{\teps^i_n(x)} \convinprob 0$ for a.e.\ $x \in \X$.
\end{theorem}
\begin{proof}
Follows from inequality~\eqref{eqn:meandeviation}, Assumptions~\ref{ass:varweaklln},~\ref{ass:errorsweaklln}, and~\ref{ass:regconsist}, Lemma~\ref{lem:varbound}, the continuous mapping theorem, and the fact that $O_p(1) O_p(1) = O_p(1)$, $O_p(1) o_p(1) = o_p(1)$, and $o_p(1) + o_p(1) = o_p(1)$.
\end{proof}

\subsection{Rates of convergence}

We refine Assumption~\ref{ass:regconsist} to obtain rates of convergence (cf.\ Assumption~6 of~\cite{kannan2020data}).

\begin{assumption}
\label{ass:regconvrate}
There is a constant\footnote{The constant $r$ is independent of~$n$, but could depend on the covariate dimension~$d_x$.} $0 < r \leq 1$ such that the regression estimates~$\hf_n$ and~$\hQ_n$ satisfy the following convergence rate criteria:
\begin{align*}
&\norm{\hf_n(x) - f^*(x)} = O_p(n^{-r/2}) \quad \text{and} \quad \norm{\hQ_n(x) - Q^*(x)} = O_p(n^{-r/2}), \quad \text{for a.e. } x \in \X, \\
&\frac{1}{n} \sum_{i=1}^{n} \norm{\hf_n(x^i) - f^*(x^i)}^2 = O_p(n^{-r}), \quad \frac{1}{n} \sum_{i=1}^{n} \bigl\lVert \bigl[\hQ_n(x^i)\bigr]^{-1} - \bigl[Q^*(x^i)\bigr]^{-1}\bigr\rVert^2 = O_p(n^{-r}).
\end{align*}
\end{assumption}

Inequality~\eqref{eqn:meandeviation} along with Assumptions~\ref{ass:varweaklln},~\ref{ass:errorsweaklln}, and~\ref{ass:regconvrate} readily yields the following result.

\begin{theorem}
Suppose Assumptions~\ref{ass:varweaklln},~\ref{ass:errorsweaklln}, and~\ref{ass:regconvrate} hold.
Then $\frac{1}{n} \sum_{i=1}^{n} \norm{\teps^i_n(x)} = O_p(n^{-r/2})$ for a.e.\ $x \in \X$.
\end{theorem}
\begin{proof}
Follows by applying Assumptions~\ref{ass:varweaklln},~\ref{ass:errorsweaklln}, and~\ref{ass:regconvrate}, Lemma~\ref{lem:varbound}, the continuous mapping theorem, and the fact that $O_p(1) +  O_p(n^{-r/2}) = O_p(1)$, $O_p(1) O_p(n^{-r/2}) = O_p(n^{-r/2})$, and $O_p(n^{-r/2}) + O_p(n^{-r/2}) = O_p(n^{-r/2})$ to inequality~\eqref{eqn:meandeviation}.
\end{proof}

\subsection{Finite sample guarantees}

We make the following additional assumptions to establish a finite sample guarantee (cf.\ Assumption~8 of~\cite{kannan2020data}).

\begin{assumption}
\label{ass:reglargedev}
The regression estimates $\hf_n$ and $\hQ_n$ possess the following finite sample properties: for any constant $\kappa > 0$, there exist positive constants $K_f(\kappa,x)$, $\bar{K}_f(\kappa)$, $\beta_f(\kappa,x)$, $\bar{\beta}_f(\kappa)$, $K_Q(\kappa,x)$, $\bar{K}_Q(\kappa)$, $\beta_Q(\kappa,x)$, and $\bar{\beta}_Q(\kappa)$ such that for each $n \in \mathbb{N}$
\begin{align*}
\mathbb{P}\bigl\{\norm{f^*(x) - \hf_n(x)} > \kappa \bigr\} &\leq K_f(\kappa,x) \exp\left(-n\beta_f(\kappa,x)\right), \quad\: \text{for a.e. } x \in \X, \\
\mathbb{P}\bigl\{\norm{Q^*(x) - \hQ_n(x)} > \kappa \bigr\} &\leq K_Q(\kappa,x) \exp\left(-n\beta_Q(\kappa,x)\right), \quad \text{for a.e. } x \in \X, \\
\mathbb{P}\biggl\{\dfrac{1}{n} \displaystyle\sum_{i=1}^{n} \norm{f^*(x^i) - \hf_n(x^i)}^2 > \kappa^2 \biggr\} &\leq \bar{K}_f(\kappa) \exp\left(-n\bar{\beta}_f(\kappa)\right), \quad \text{and} \\
\mathbb{P}\biggl\{\frac{1}{n} \sum_{i=1}^{n} \bigl\lVert \bigl[\hQ_n(x^i)\bigr]^{-1} - \bigl[Q^*(x^i)\bigr]^{-1}\bigr\rVert^2 > \kappa^2 \biggr\} &\leq \bar{K}_Q(\kappa) \exp\left(-n\bar{\beta}_Q(\kappa)\right).
\end{align*}
\end{assumption}

The next two assumptions strengthen Assumptions~\ref{ass:varweaklln} and~\ref{ass:errorsweaklln} to assume finite sample properties for the quantities involved.

\begin{assumption}
\label{ass:varlargedev}
For any constant $\kappa > 0$, there exist positive constants $\gamma_{Q}(\kappa)$ and $\bar{\gamma}_{Q}(\kappa)$ such that for each $n \in \mathbb{N}$
\begin{align*}
\mathbb{P}\biggl\{ \biggl(\frac{1}{n}\sum_{i=1}^{n} \bigl\lVert \bigl[Q^*(x^i)\bigr]^{-1}\bigr\rVert^2\biggr)^{1/2} > \Bigl(\expect{\bigl\lVert \bigl[Q^*(X)\bigr]^{-1} \bigr\rVert^2}\Bigr)^{1/2} + \kappa \biggr\} &\leq \exp(-n \gamma_{Q}(\kappa)), \\
\mathbb{P}\biggl\{ \biggl(\frac{1}{n} \sum_{i=1}^{n} \norm{Q^*(x^i)}^4\biggr)^{1/4} > \bigl(\expect{\norm{Q^*(X)}^4}\bigr)^{1/4} + \kappa \biggr\} &\leq \exp(-n \bar{\gamma}_{Q}(\kappa)).
\end{align*}
\end{assumption}

\begin{assumption}
\label{ass:errorlargedev}
For any constant $\kappa > 0$, there exist positive constants $\gamma_{\varepsilon}(\kappa)$ and $\bar{\gamma}_{\varepsilon}(\kappa)$ such that for each $n \in \mathbb{N}$
\[
\mathbb{P}\biggl\{ \frac{1}{n}\sum_{i=1}^{n} \norm{\varepsilon^i} > \expect{\norm{\varepsilon}} + \kappa \biggr\} \leq \exp(-n \gamma_{\varepsilon}(\kappa)), \quad \mathbb{P}\biggl\{ \biggl(\frac{1}{n}\sum_{i=1}^{n} \norm{\varepsilon^i}^4\biggr)^{1/4} > (\expect{\norm{\varepsilon}^4})^{1/4} + \kappa \biggr\} \leq \exp(-n \bar{\gamma}_{\varepsilon}(\kappa)).
\]
\end{assumption}

The first part of Assumption~\ref{ass:varlargedev} holds, e.g., if for each $\kappa > 0$, there is a constant $\gamma_Q(\kappa) > 0$ such that
\[
\mathbb{P}\biggl\{ \frac{1}{n}\sum_{i=1}^{n} \bigl\lVert \bigl[Q^*(x^i)\bigr]^{-1}\bigr\rVert^2 > \expect{\bigl\lVert \bigl[Q^*(X)\bigr]^{-1} \bigr\rVert^2} + \kappa^2 \biggr\} \leq \exp(-n \gamma_{Q}(\kappa)).
\]
The function $\gamma_Q(\cdot)$ in the inequality above is related to the so-called rate function in large deviations theory (see Section~7.2.8 of~\cite{shapiro2009lectures}).
Similar conclusions hold for the probability inequalities involving the terms $\frac{1}{n} \sum_{i=1}^{n} \norm{Q^*(x^i)}^4$ and $\frac{1}{n}\sum_{i=1}^{n} \norm{\varepsilon^i}^4$ in Assumptions~\ref{ass:varlargedev} and~\ref{ass:errorlargedev}.
From large deviations theory, we can also conclude that the constants $\gamma_Q(\kappa)$, $\bar{\gamma}_Q(\kappa)$, $\gamma_{\varepsilon}(\kappa)$, and~$\bar{\gamma}_{\varepsilon}(\kappa)$ in Assumptions~\ref{ass:varlargedev} and~\ref{ass:errorlargedev} are guaranteed to exist for i.i.d.\ data~$\D_n$ and for each constant $\kappa > 0$ whenever the following light-tail conditions hold: $\mathbb{E}\bigr[\exp\bigl(\bigl\lVert \bigl[Q^*(X)\bigr]^{-1} \bigr\rVert^p\bigr)\bigr] < +\infty$ for some $p > 2$, $\mathbb{E}[\exp(\norm{Q^*(X)}^p)] < +\infty$ for some $p > 4$, and $\mathbb{E}[\exp(\norm{\varepsilon}^p)] < +\infty$ for some $p > 4$.
The discussion following Assumption~7 of~\cite{kannan2020data} provides avenues for verifying Assumptions~\ref{ass:varlargedev} and~\ref{ass:errorlargedev} for non-i.i.d.\ data~$\D_n$.

We are now ready to state our finite sample guarantee.

\begin{theorem}
Suppose Assumptions~\ref{ass:reglargedev},~\ref{ass:varlargedev}, and~\ref{ass:errorlargedev} hold.
Then, for any constant $\kappa > 0$ and a.e.\ $x \in \X$, there exist positive constants $\tilde{K}(\kappa,x)$ and $\tilde{\beta}(\kappa,x)$ such that
\begin{align*}
\mathbb{P}\biggl\{\dfrac{1}{n} \displaystyle\sum_{i=1}^{n} \norm{\teps^i_{n}(x)} > \kappa \biggr\} &\leq \tilde{K}(\kappa,x) \exp(-n\tilde{\beta}(\kappa,x)).
\end{align*}
\end{theorem}
\begin{proof}
Using~\eqref{eqn:meandeviation} and the inequality $\mathbb{P}\{V + W > c_1 + c_2\} \leq \mathbb{P}\{V > c_1\} + \mathbb{P}\{W > c_2\}$ for any random variables $V$, $W$ and constants $c_1$, $c_2$, we get
\begin{align}
\label{eqn:meandeviation_finitesamp}
&\mathbb{P}\Bigl\{\dfrac{1}{n} \displaystyle\sum_{i=1}^{n} \norm{\teps^i_{n}(x)} > \kappa \Bigr\} \nonumber\\
\leq& \mathbb{P}\Bigl\{\norm{\hf_n(x) - f^*(x)} > \frac{\kappa}{4}\Bigr\} + \mathbb{P}\biggl\{\biggl(\frac{1}{n} \sum_{i=1}^{n} \norm{\varepsilon^i}\biggr) \norm{\hQ_n(x) - Q^*(x)} > \frac{\kappa}{4}\biggr\} + \nonumber \\
&\quad \mathbb{P}\biggl\{ \norm{\hQ_n(x)} \biggl(\frac{1}{n} \sum_{i=1}^{n} \bigl\lVert \bigl[\hQ_n(x^i)\bigr]^{-1} - \bigl[Q^*(x^i)\bigr]^{-1}\bigr\rVert^2\biggr)^{1/2} \biggl(\frac{1}{n} \sum_{i=1}^{n} \norm{Q^*(x^i)}^4\biggr)^{1/4} \biggl(\frac{1}{n} \sum_{i=1}^{n} \norm{\varepsilon^i}^4\biggr)^{1/4} > \frac{\kappa}{4} \biggr\} + \nonumber \\
&\quad \mathbb{P}\biggl\{ \norm{\hQ_n(x)} \biggl(\frac{1}{n} \sum_{i=1}^{n} \bigl\lVert \bigl[\hQ_n(x^i)\bigr]^{-1}\bigr\rVert^2\biggr)^{1/2} \biggl(\frac{1}{n} \sum_{i=1}^{n} \norm{f^*(x^i) - \hf_n(x^i)}^2\biggr)^{1/2} > \frac{\kappa}{4}\biggr\}.
\end{align}
For a.e.\ $x \in \X$, the first term on the r.h.s.\ of inequality~\eqref{eqn:meandeviation_finitesamp} can be bounded using Assumption~\ref{ass:reglargedev} as
\begin{align*}
&\mathbb{P}\Bigl\{\norm{\hf_n(x) - f^*(x)} > \frac{\kappa}{4}\Bigr\} \leq K_f(\tfrac{\kappa}{4},x) \exp(-n\beta_f(\tfrac{\kappa}{4},x)).
\end{align*}

Next, consider the second term on the r.h.s.\ of inequality~\eqref{eqn:meandeviation_finitesamp}.
We have for a.e.\ $x \in \X$
\begin{align*}
&\mathbb{P}\biggl\{\biggl(\frac{1}{n} \sum_{i=1}^{n} \norm{\varepsilon^i}\biggr) \norm{\hQ_n(x) - Q^*(x)} > \frac{\kappa}{4}\biggr\} \\
\leq& \mathbb{P}\biggl\{\frac{1}{n} \sum_{i=1}^{n} \norm{\varepsilon^i} > \mathbb{E}[\norm{\varepsilon}] + \kappa\biggr\} + \mathbb{P}\Bigl\{(\mathbb{E}[\norm{\varepsilon}] + \kappa) \norm{\hQ_n(x) - Q^*(x)} > \frac{\kappa}{4}\Bigr\} \\
\leq& \exp(-n\gamma_{\varepsilon}(\kappa)) + \mathbb{P}\biggl\{ \norm{\hQ_n(x) - Q^*(x)} > \frac{\kappa}{4(\mathbb{E}[\norm{\varepsilon}] + \kappa)}\biggr\} \\
\leq& \exp(-n\gamma_{\varepsilon}(\kappa)) + K_Q\bigl(\tfrac{\kappa}{4(\mathbb{E}[\norm{\varepsilon}] + \kappa)},x\bigr) \exp\bigl(-n\beta_Q(\tfrac{\kappa}{4(\mathbb{E}[\norm{\varepsilon}] + \kappa)},x)\bigr),
\end{align*}
where the second inequality follows from Assumption~\ref{ass:errorlargedev} and the final step follows from Assumption~\ref{ass:reglargedev}.

The third term on the r.h.s.\ of inequality~\eqref{eqn:meandeviation_finitesamp} can be bounded for a.e.\ $x \in \X$ as
\begin{align*}
&\mathbb{P}\biggl\{ \norm{\hQ_n(x)} \biggl(\frac{1}{n} \sum_{i=1}^{n} \bigl\lVert \bigl[\hQ_n(x^i)\bigr]^{-1} - \bigl[Q^*(x^i)\bigr]^{-1}\bigr\rVert^2\biggr)^{1/2} \biggl(\frac{1}{n} \sum_{i=1}^{n} \norm{Q^*(x^i)}^4\biggr)^{1/4} \biggl(\frac{1}{n} \sum_{i=1}^{n} \norm{\varepsilon^i}^4\biggr)^{1/4} > \frac{\kappa}{4} \biggr\} \\
\leq& \mathbb{P}\bigl\{ \norm{\hQ_n(x)} > \norm{Q^*(x)} + \kappa \bigr\} + \mathbb{P}\biggl\{ \biggl(\frac{1}{n} \sum_{i=1}^{n} \norm{Q^*(x^i)}^4\biggr)^{1/4} > \bigl(\mathbb{E}[\norm{Q^*(X)}^4]\bigr)^{1/4} + \kappa\biggr\} + \\
&\quad \mathbb{P}\biggl\{ \biggl(\frac{1}{n} \sum_{i=1}^{n} \norm{\varepsilon^i}^4\biggr)^{1/4} > \bigl(\mathbb{E}[\norm{\varepsilon}^4]\bigr)^{1/4} + \kappa\biggr\} + \\
&\quad \mathbb{P}\biggl\{ \bigl( \norm{Q^*(x)} + \kappa \bigr) \bigl(\bigl(\mathbb{E}[\norm{Q^*(X)}^4]\bigr)^{1/4} + \kappa\bigr) \bigl(\bigl(\mathbb{E}[\norm{\varepsilon}^4]\bigr)^{1/4} + \kappa\bigr) \biggl(\frac{1}{n} \sum_{i=1}^{n} \bigl\lVert \bigl[\hQ_n(x^i)\bigr]^{-1} - \bigl[Q^*(x^i)\bigr]^{-1}\bigr\rVert^2\biggr)^{1/2} > \frac{\kappa}{4}\biggr\} \\
\leq& K_Q(\kappa,x) \exp\left(-n\beta_Q(\kappa,x)\right) + \exp(-n\bar{\gamma}_{Q}(\kappa)) + \exp(-n\bar{\gamma}_{\varepsilon}(\kappa)) + \\
&\quad \mathbb{P}\biggl\{ \biggl(\frac{1}{n} \sum_{i=1}^{n} \bigl\lVert \bigl[\hQ_n(x^i)\bigr]^{-1} - \bigl[Q^*(x^i)\bigr]^{-1}\bigr\rVert^2\biggr)^{1/2} > h_1(\kappa,x)\biggr\} \\
\leq& K_Q(\kappa,x) \exp\left(-n\beta_Q(\kappa,x)\right) + \exp(-n\bar{\gamma}_{Q}(\kappa)) + \exp(-n\bar{\gamma}_{\varepsilon}(\kappa)) + \bar{K}_Q(h_1(\kappa,x)) \exp(-n\bar{\beta}_Q(h_1(\kappa,x))),
\end{align*}
where the second inequality follows from Assumptions~\ref{ass:reglargedev},~\ref{ass:varlargedev}, and~\ref{ass:errorlargedev}, the final inequality follows from Assumption~\ref{ass:reglargedev}, and
\[
h_1(\kappa,x) := \frac{\kappa}{4\bigl( \norm{Q^*(x)} + \kappa \bigr) \bigl(\bigl(\mathbb{E}[\norm{Q^*(X)}^4]\bigr)^{1/4} + \kappa\bigr) \bigl(\bigl(\mathbb{E}[\norm{\varepsilon}^4]\bigr)^{1/4} + \kappa\bigr)}.
\]

Finally, the fourth term on the r.h.s.\ of inequality~\eqref{eqn:meandeviation_finitesamp} can be bounded for a.e.\ $x \in \X$ as
\begin{align*}
&\mathbb{P}\biggl\{ \norm{\hQ_n(x)} \biggl(\frac{1}{n} \sum_{i=1}^{n} \bigl\lVert \bigl[\hQ_n(x^i)\bigr]^{-1}\bigr\rVert^2\biggr)^{1/2} \biggl(\frac{1}{n} \sum_{i=1}^{n} \norm{f^*(x^i) - \hf_n(x^i)}^2\biggr)^{1/2} > \frac{\kappa}{4}\biggr\} \\
\leq& \mathbb{P}\bigl\{ \norm{\hQ_n(x)} > \norm{Q^*(x)} + \kappa \bigr\} + \mathbb{P}\biggl\{ \biggl(\frac{1}{n} \sum_{i=1}^{n} \bigl\lVert \bigl[\hQ_n(x^i)\bigr]^{-1}\bigr\rVert^2\biggr)^{1/2} > \Bigl(\mathbb{E}\bigl[ \bigl\lVert \bigl[Q^*(X)\bigr]^{-1} \bigr\rVert^2 \bigr]\Bigr)^{1/2} + 2\kappa\biggr\} + \\
&\quad \mathbb{P}\biggl\{ \bigl( \norm{Q^*(x)} + \kappa \bigr) \Bigl(\Bigl(\mathbb{E}\bigl[ \bigl\lVert \bigl[Q^*(X)\bigr]^{-1} \bigr\rVert^2 \bigr]\Bigr)^{1/2} + 2\kappa\Bigr) \biggl(\frac{1}{n} \sum_{i=1}^{n} \norm{f^*(x^i) - \hf_n(x^i)}^2\biggr)^{1/2} > \frac{\kappa}{4}\biggr\} \\
\leq& K_Q(\kappa,x) \exp\left(-n\beta_Q(\kappa,x)\right) + \mathbb{P}\biggl\{ \biggl(\frac{1}{n} \sum_{i=1}^{n} \norm{f^*(x^i) - \hf_n(x^i)}^2\biggr)^{1/2} > h_2(\kappa,x)\biggr\} + \\
&\quad\mathbb{P}\biggl\{ \biggl(\frac{1}{n} \sum_{i=1}^{n} \bigl\lVert \bigl[\hQ_n(x^i)\bigr]^{-1} - \bigl[Q^*(x^i)\bigr]^{-1}\bigr\rVert^2\biggr)^{1/2} + \biggl(\frac{1}{n} \sum_{i=1}^{n} \bigl\lVert \bigl[Q^*(x^i)\bigr]^{-1}\bigr\rVert^2\biggr)^{1/2} > \Bigl(\mathbb{E}\bigl[ \bigl\lVert \bigl[Q^*(X)\bigr]^{-1} \bigr\rVert^2 \bigr]\Bigr)^{1/2} + 2\kappa\biggr\} \\
\leq& K_Q(\kappa,x) \exp\left(-n\beta_Q(\kappa,x)\right) + \bar{K}_f(h_2(\kappa,x)) \exp(-n\bar{\beta}_f(h_2(\kappa,x))) + \bar{K}_Q(\kappa) \exp(-n\bar{\beta}_Q(\kappa)) + \exp(-n\gamma_{Q}(\kappa)),
\end{align*}
where the second inequality follows from Assumption~\ref{ass:reglargedev} and Lemma~\ref{lem:varbound}, the final inequality follows from Assumptions~\ref{ass:reglargedev} and~\ref{ass:varlargedev} and the probability inequality stated at the beginning of this proof, and
\[
h_2(\kappa,x) := \frac{\kappa}{4\bigl( \norm{Q^*(x)} + \kappa \bigr) \Bigl(\Bigl(\mathbb{E}\bigl[ \bigl\lVert \bigl[Q^*(X)\bigr]^{-1} \bigr\rVert^2 \bigr]\Bigr)^{1/2} + 2\kappa\Bigr)}.
\]
Putting the above bounds together in inequality~\eqref{eqn:meandeviation_finitesamp}, we have for a.e.\ $x \in \X$
\begin{align}
\label{eqn:meandeviation_finitesamp2}
\mathbb{P}\biggl\{\dfrac{1}{n} \displaystyle\sum_{i=1}^{n} \norm{\teps^i_{n}(x)} > \kappa \biggr\} &\leq \exp(-n\gamma_{\varepsilon}(\kappa)) + \exp(-n\bar{\gamma}_{\varepsilon}(\kappa)) + \exp(-n\gamma_{Q}(\kappa)) + \exp(-n\bar{\gamma}_{Q}(\kappa)) + \\
&\quad K_f(\tfrac{\kappa}{4},x) \exp(-n\beta_f(\tfrac{\kappa}{4},x)) + \bar{K}_f(h_2(\kappa,x)) \exp(-n\bar{\beta}_f(h_2(\kappa,x))) + \nonumber \\
&\quad K_Q\bigl(\tfrac{\kappa}{4(\mathbb{E}[\norm{\varepsilon}] + \kappa)},x\bigr) \exp\bigl(-n\beta_Q(\tfrac{\kappa}{4(\mathbb{E}[\norm{\varepsilon}] + \kappa)},x)\bigr) + 2K_Q(\kappa,x) \exp\left(-n\beta_Q(\kappa,x)\right) + \nonumber\\
&\quad \bar{K}_Q(\kappa) \exp(-n\bar{\beta}_Q(\kappa)) + \bar{K}_Q(h_1(\kappa,x)) \exp(-n\bar{\beta}_Q(h_1(\kappa,x))) \nonumber,
\end{align}
which then implies the desired result.
\end{proof}

\vspace*{0.05in}

Suppose we make the mild assumptions that the functions $\bar{K}_f(\cdot)$, $K_Q(\cdot,x)$, and $\bar{K}_Q(\cdot)$ in Assumption~\ref{ass:reglargedev} are monotonically nonincreasing on $\mathbb{R}_+$ and the functions $\bar{\beta}_f(\cdot)$, $\beta_Q(\cdot,x)$, and $\bar{\beta}_Q(\cdot)$ therein are monotonically nondecreasing on $\mathbb{R}_+$ (cf.\ Appendix~EC.3 of~\cite{kannan2020data}).
For a.e.\ $x \in \X$ and tolerance $\kappa$ satisfying
\[
\kappa < \min\Bigl\{\mathbb{E}[\norm{\varepsilon}], \norm{Q^*(x)}, \bigl(\mathbb{E}[\norm{Q^*(X)}^4]\bigr)^{1/4}, \bigl(\mathbb{E}\bigl[ \bigl\lVert\bigl[Q^*(X)\bigr]^{-1} \bigr\rVert^2 \bigr]\bigr)^{1/2}\Bigr\},
\]
we can use inequality~\eqref{eqn:meandeviation_finitesamp2} to derive the bound
\begin{align*}
\mathbb{P}\biggl\{\dfrac{1}{n} \displaystyle\sum_{i=1}^{n} \norm{\teps^i_{n}(x)} > \kappa \biggr\} &\leq \exp(-n\gamma_{\varepsilon}(\kappa)) + \exp(-n\bar{\gamma}_{\varepsilon}(\kappa)) + \exp(-n\gamma_{Q}(\kappa)) + \exp(-n\bar{\gamma}_{Q}(\kappa)) + \\
&\quad K_f(\tfrac{\kappa}{4},x) \exp(-n\beta_f(\tfrac{\kappa}{4},x)) + \bar{K}_f(\bar{h}_2(\kappa,x)) \exp(-n\bar{\beta}_f(\bar{h}_2(\kappa,x))) + \nonumber \\
&\quad K_Q\bigl(\tfrac{\kappa}{8\mathbb{E}[\norm{\varepsilon}]},x\bigr) \exp\bigl(-n\beta_Q(\tfrac{\kappa}{8\mathbb{E}[\norm{\varepsilon}]},x)\bigr) + 2K_Q(\kappa,x) \exp\left(-n\beta_Q(\kappa,x)\right) + \nonumber\\
&\quad \bar{K}_Q(\kappa) \exp(-n\bar{\beta}_Q(\kappa)) + \bar{K}_Q(\bar{h}_1(\kappa,x)) \exp(-n\bar{\beta}_Q(\bar{h}_1(\kappa,x))), \nonumber
\end{align*}
where
\[
\bar{h}_1(\kappa,x) := \frac{\kappa}{32\norm{Q^*(x)} \bigl(\mathbb{E}[\norm{Q^*(X)}^4]\bigr)^{1/4} \bigl(\mathbb{E}[\norm{\varepsilon}^4]\bigr)^{1/4}}, \quad \bar{h}_2(\kappa,x) := \frac{\kappa}{24\norm{Q^*(x)}\Bigl(\mathbb{E}\bigl[ \bigl\lVert \bigl[Q^*(X)\bigr]^{-1} \bigr\rVert^2 \bigr]\Bigr)^{1/2}}.
\]
Therefore, for a.e.\ $x \in \X$ and $\kappa < \min\bigl\{\mathbb{E}[\norm{\varepsilon}], \norm{Q^*(x)}, \bigl(\mathbb{E}[\norm{Q^*(X)}^4]\bigr)^{1/4}, \bigl(\mathbb{E}\bigl[ \bigl\lVert\bigl[Q^*(X)\bigr]^{-1} \bigr\rVert^2 \bigr]\bigr)^{1/2}\bigr\}$
\[
\mathbb{P}\biggl\{\dfrac{1}{n} \displaystyle\sum_{i=1}^{n} \norm{\teps^i_{n}(x)} > \kappa \biggr\} \leq \tilde{K}(\kappa,x) \exp(-n\tilde{\beta}(\kappa,x)),
\]
with $\tilde{K}(\kappa,x) := 4 + K_f(\tfrac{\kappa}{4},x) + \bar{K}_f(\bar{h}_2(\kappa,x)) + K_Q\bigl(\tfrac{\kappa}{8\mathbb{E}[\norm{\varepsilon}]},x\bigr) + 2K_Q(\kappa,x) + \bar{K}_Q(\kappa) + \bar{K}_Q(\bar{h}_1(\kappa,x))$ and
\begin{align*}
\tilde{\beta}(\kappa,x) &:= \min\Bigr\{ \gamma_{\varepsilon}(\kappa), \bar{\gamma}_{\varepsilon}(\kappa), \gamma_{Q}(\kappa), \bar{\gamma}_{Q}(\kappa), \beta_f(\tfrac{\kappa}{4},x), \bar{\beta}_f(\bar{h}_2(\kappa,x)), \beta_Q\bigl(\tfrac{\kappa}{8\mathbb{E}[\norm{\varepsilon}]},x\bigr), \beta_Q(\kappa,x), \bar{\beta}_Q(\kappa), \bar{\beta}_Q(\bar{h}_1(\kappa,x)) \Bigr\}.
\end{align*}
Unlike the functions $h_1(\cdot,x)$ and $h_2(\cdot,x)$, the functions $\bar{h}_1(\cdot,x)$ and $\bar{h}_2(\cdot,x)$ are linear.
Consequently, for small-enough tolerances~$\kappa > 0$ and a given risk level $\alpha \in (0,1)$, the above simplification enables an easier interpretation of the sample size~$n$ required for $\mathbb{P}\bigl\{\frac{1}{n} \sum_{i=1}^{n} \norm{\teps^i_n(x)} > \kappa\bigr\} \leq \alpha$.
This can in turn enable a more interpretable estimate of the sample size $n$ required for solutions of the ER-SAA problem~\eqref{eqn:app} and the ER-DRO problem~\eqref{eqn:dro} to be approximately optimal to the true problem~\eqref{eqn:speq} with probability $1-\alpha$ (cf.\ Proposition~2 of~\cite{kannan2020data}).

\section{Some regression setups that satisfy our assumptions}
\label{sec:regr}

In this section, we verify that Assumptions~\ref{ass:regconsist},~\ref{ass:regconvrate}, and~\ref{ass:reglargedev} hold for some regression setups.
We do not attempt to be exhaustive.
We first discuss methods for estimating the regression function~$f^*$ and note their asymptotic and finite sample guarantees.
We then list some popular models for the class of functions~$\Q$, discuss approaches for estimating the matrix-valued function~$Q^*$, and note their theoretical guarantees.

\subsection{Estimating the regression function}
\label{subsec:estf*}

We identify conditions under which the parts of Assumptions~\ref{ass:regconsist},~\ref{ass:regconvrate}, and~\ref{ass:reglargedev} involving the regression estimate~$\hf_n$ hold for some prediction setups. 
Although these assumptions on~$\hf_n$ are the same as those in Assumptions~4,~6, and~8 of~\cite{kannan2020data}, we focus on regression setups that work in the heteroscedastic setting.

\paragraph{Ordinary least squares (OLS) regression.}
When the regression function~$f^*$ is linear, its OLS estimate $\hf_n$ satisfies Assumptions~\ref{ass:regconsist} and~\ref{ass:regconvrate} with constant $r = 1$ (see Proposition~EC.3.\ of~\cite{kannan2020data} for details).
Furthermore, Theorem~11 and Remark~12 of~\cite{hsu2012random} can be used to readily identify conditions under which the estimates~$\hf_n$ possess a finite sample guarantee like in Assumption~\ref{ass:reglargedev}.
However, OLS regression does not yield an efficient estimator\footnote{Here, by the term {\it efficient estimator}, we mean a minimum variance unbiased estimator.} of~$f^*$ in the heteroscedastic case~\cite{romano2017resurrecting}.
An alternative to OLS regression is feasible weighted least squares (FWLS) regression~\cite{romano2017resurrecting,robinson1987asymptotically}, which results in asymptotically efficient estimates when the estimate~$\hQ_n$ of~$Q^*$ is consistent.
These asymptotic results of FWLS regression continue to hold at the expense of asymptotic efficiency even if the estimate~$\hQ_n$ of~$Q^*$ may be inconsistent (see, e.g., Section~3.3 of~\cite{romano2017resurrecting}).

\paragraph{Sparse regression methods.}
Proposition~EC.4.\ of~\cite{kannan2020data} lists conditions under which the ordinary Lasso regression estimate~$\hf_n$ satisfies Assumptions~\ref{ass:regconsist} and~\ref{ass:regconvrate} with constant $r = 1$ and a finite sample guarantee like in Assumption~\ref{ass:reglargedev}.
Theorem~1 of~\citet{belloni2012sparse} outlines conditions under which similar asymptotic results hold for the heteroscedasticity-adapted Lasso.
\citet{medeiros2016l1} and~\citet{ziel2016iteratively} present asymptotic analyses of the adaptive Lasso for time series data.
Their analyses applies to GARCH-type processes.
Theorems~2 and~3 of~\cite{medeiros2016l1} and Theorem~1 of~\cite{ziel2016iteratively} present conditions under which the estimate~$\hf_n$ satisfies Assumptions~\ref{ass:regconsist} and~\ref{ass:regconvrate} with $r = 1$.
\citet{belloni2014pivotal} present asymptotic and finite sample guarantees for the heteroscedasticity-adapted square-root Lasso.
Finally,~\citet{dalalyan2013learning} introduce a scaled heteroscedastic Dantzig selector.
Theorem~5.2 therein presents large deviation bounds for both regression estimates~$\hf_n$ and~$\hQ_n$ under certain sparsity assumptions\footnote{Although~\cite{dalalyan2013learning} consider the fixed design setting, their analysis can be modified to accommodate random designs under suitable assumptions on the distribution~$P_X$ of the covariates~$X$ (see Section~4 of~\cite{dalalyan2013learning}).}.

\paragraph{Other M-estimators.}
The conclusions for OLS regression carry over to more general M-estimators.
In particular, Appendix~EC.2 of~\cite{kannan2020data} presents conditions under which Assumptions~\ref{ass:regconsist} and~\ref{ass:regconvrate} continue to hold with $r = 1$.
Similar to the special case of OLS regression, vanilla M-estimators may no longer be efficient---feasible weighted M-estimation is an asymptotically efficient alternative.
Theorems~1,~3, and~5 of~\citet{sun2020adaptive} and Theorem~2.1 of~\citet{zhou2018new} present large deviation results of the form Assumption~\ref{ass:reglargedev} for adaptive Huber regression when the function~$f^*$ is linear.
Remarkably, their results hold even for heavy-tailed error distributions.
Finally,~\citet{schick1996weighted} considers a semiparametric regression setup for~$f^*$ and establishes rates of convergence of weighted least squares estimates.

\paragraph{kNN regression.}
Proposition~EC.5.\ of~\cite{kannan2020data} summarizes conditions under which the kNN regression estimate~$\hf_n$ of~$f^*$ satisfies Assumptions~\ref{ass:regconsist} and~\ref{ass:regconvrate} with constant $r = O(1)/d_x$.
It also notes conditions under which~$\hf_n$ possesses a finite sample guarantee like in Assumption~\ref{ass:reglargedev} (cf.\ Corollary~1 of~\cite{jiang2019non}).

\paragraph{Kernel regression.}
\citet{hansen2008uniform} studies conditions under which kernel regression estimates are uniformly consistent given dependent data~$\D_n$ satisfying mixing conditions.
Theorems~1,~2, and~4 therein can be used to show that the kernel regression estimate~$\hf_n$ satisfies Assumptions~\ref{ass:regconsist} and~\ref{ass:regconvrate} with constant $r = O(1)/d_x$.
\citet{mokkadem2008large} study large deviations results for some kernel regression estimates.

\subsection{Estimating the conditional covariance matrix of the errors}

In this section, we identify conditions under which the parts of Assumptions~\ref{ass:regconsist},~\ref{ass:regconvrate}, and~\ref{ass:reglargedev} involving the regression estimate~$\hQ_n$ hold for some prediction setups. 
These assumptions for $\hQ_n$---in particular, Assumption~\ref{ass:reglargedev}---are not as well-studied in the literature as those for~$\hf_n$. Therefore, they are typically harder to verify than their counterparts in Section~\ref{subsec:estf*}.
Because deriving theoretical properties of estimators for the heteroscedastic setting and deriving finite sample properties of estimators in general are areas of topical interest, we envision that future research will enable easier verification of these assumptions.

For simplicity, we only consider function classes~$\Q$ that comprise diagonal covariance matrices (cf.\ \cite{zhou2018new}), although the theoretical developments in Section~\ref{sec:bounds} apply more generally.
Bauwens et al.\ \cite{bauwens2006multivariate} review some model classes~$\Q$ with non-diagonal covariance matrices that are popular in time series modeling.

\begin{example}[Parametric Models]
\label{exm:paramcov}
The model class is 
\[
\Q = \{Q : \R^{d_x} \to \R^{d_y \times d_y} : Q(X) = \text{diag}(q_1(X), q_2(X), \dots, q_{d_y}(X))\},
\]
where $q_j : \R^{d_x} \to \R_+$ for each $j \in [d_y]$.
Forms of the functions $q_j$ of interest include~\cite{powell2010models,romano2017resurrecting}:
\begin{enumerate}[label=\roman*.]

\item $(q_j(X))^2 = \sigma^2_j (1+ \tr{\theta}_j X)^2$ for parameters $(\sigma_j,\theta_j)$, 

\item $(q_j(X))^2 = \exp(\sigma_j + \tr{\theta}_j X)$ for parameters $(\sigma_j,\theta_j)$,

\item $(q_j(X))^2 = \exp\bigl(\sigma_j + \tr{\theta}_j \log(X)\bigr)$ for parameters $(\sigma_j,\theta_j)$. 
\end{enumerate}
For the rest of the note, we absorb the parameter~$\sigma_j$ into~$\theta_j$ for simplicity of presentation. With this notation, the above examples can be cast in the general form $(q_j(X))^2 = h_j(\tr{\theta}_j g_j(X))$ for known functions $h_j$ and $g_j$ and a parameter~$\theta_j$ that is to be estimated.
The above setup can also accommodate cases where the parameters of the function~$Q^*$ include some of the parameters of the function~$f^*$. 
\end{example}

\begin{example}[Nonparametric Model]
\label{exm:nonparamcov}
The model class is 
\[
\Q = \{Q : \R^{d_x} \to \R^{d_y \times d_y} : Q(X) = \text{diag}(q_1(X), q_2(X), \dots, q_{d_y}(X))\},
\]
where each $q_j : \R^{d_x} \to \R_+$ is assumed to be     `sufficiently smooth'.
Chapter~8 of~\citet{fan2008nonlinear} presents some popular models for the functions~$q_j$ in a time series context.
\end{example}

Suppose for ease of exposition that the covariance matrix of the errors~$\varepsilon$ is the identity matrix.
Then, for each $j \in [d_y]$ and any $\bar{x} \in \X$, we have $\expect{(Y_j - f^*_j(X))^2 \mid X = \bar{x}} = (q^*_j(\bar{x}))^2$ for the components $q^*_j(\bar{x})$ of $Q^*(\bar{x})$ in Examples~\ref{exm:paramcov} and~\ref{exm:nonparamcov}.
This motivates the estimation of each function~$q^*_j$ by regressing the squared residuals $(y^i_j - \hf_{j,n}(x^i))^2$ on the covariate observation~$x^i$.
For the parametric setup in Example~\ref{exm:paramcov}, this nonlinear regression problem can often be transformed into a linear regression problem.
An alternative for the parametric regression setup is to estimate the parameters~$\theta_j$ in~$\hQ_n$ concurrently with the parameters of the estimate~$\hf_n$ using an M-estimation procedure.
Section~3 of~\citet{davidian1987variance} outlines several approaches for estimating the parameters in Example~\ref{exm:paramcov}, including the methods mentioned above.
Chapter~8 of~\citet{fan2008nonlinear} discusses nonparametric regression methods for estimating each function~$q_j$.

We now outline approaches for verifying that the estimate~$\hQ_n$ satisfies Assumptions~\ref{ass:regconsist},~\ref{ass:regconvrate}, and~\ref{ass:reglargedev}.
Consider first the parametric setup in Example~\ref{exm:paramcov}.
Suppose the function $Q^*(\cdot) \equiv Q(\cdot;\theta^*)$ for some function~$Q$ and the goal is to estimate the parameter $\theta^*$.
Let~$\hth_n$ denote the estimate of~$\sth$ corresponding to the regression estimate~$\hQ_n$, i.e., $\hQ_n(\cdot) \equiv Q(\cdot;\hth_n)$.
Suppose for a.e.\ realization~$x \in \X$, the function~$Q(x;\cdot)$ is Lipschitz continuous with Lipschitz constant~$L_Q(x)$ and its inverse~$[Q(x;\cdot)]^{-1}$ is also Lipschitz continuous with Lipschitz constant~$\bar{L}_Q(x)$.
These assumptions hold for the model classes in Example~\ref{exm:paramcov} if the parameters~$\theta$ therein are restricted to lie in suitable compact sets\footnote{As noted in~\cite[Appendix~EC.3.2.]{kannan2020data}, it suffices to assume that the above Lipschitz continuity assumptions hold locally for the asymptotic results.}.
Because
\[
\norm{\hQ_n(x) - Q^*(x)} \leq L_Q(x)\norm{\hth_n - \sth}\  \text{ and } \ \  \frac{1}{n} \sum_{i=1}^{n} \bigl\lVert \bigl[\hQ_n(x^i)\bigr]^{-1} - \bigl[Q^*(x^i)\bigr]^{-1}\bigr\rVert^2 \leq \Bigl(\frac{1}{n} \sum_{i=1}^{n}\bar{L}^2_Q(x^i)\Bigr) \norm{\hth_n - \sth}^2,
\]
asymptotic and finite sample guarantees on the estimator~$\hth_n$ of~$\sth$ directly translate to the asymptotic and finite sample guarantees on the estimate~$\hQ_n$ in Assumptions~\ref{ass:regconsist},~\ref{ass:regconvrate}, and~\ref{ass:reglargedev}.
When the functions~$f^*$ and~$Q^*$ are jointly estimated using M-estimators, the results listed in Appendix~EC.3.2.\ of~\cite{kannan2020data} provide conditions under which the estimator~$\hth_n$ of~$\sth$ is consistent and Assumptions~\ref{ass:regconsist} and~\ref{ass:regconvrate} hold with $r = 1$.
They also present a hard-to-verify uniform exponential bound condition under which~$\hth_n$ possesses a finite sample guarantee.
\citet{carroll1982robust} consider robust M-estimators for~$\sth$ that possess a similar rate of convergence when~$f^*$ is linear.
\citet{dalalyan2013learning} present asymptotic and finite sample guarantees for a scaled Dantzig estimator of~$\sth$ under some sparsity assumptions.
Finally,~\citet{fan2014quasi} present a quasi-maximum likelihood approach for estimating the parameters of GARCH models and investigate their asymptotic properties.

Next, consider the nonparametric setup in Example~\ref{exm:nonparamcov}, and suppose the function~$Q^*$ and its regression estimate~$\hQ_n$ are (asymptotically) a.s.\ uniformly invertible\footnote{Although inequality~\eqref{eqn:meandeviation_alt} in Section~\ref{sec:bounds} can yield similar guarantees under such uniform invertibility assumptions, we stick with Assumptions~\ref{ass:regconsist},~\ref{ass:regconvrate}, and~\ref{ass:reglargedev} dictated by inequality~\eqref{eqn:meandeviation} for simplicity.}.
We have
\begin{align*}
&\frac{1}{n} \sum_{i=1}^{n} \bigl\lVert \bigl[\hQ_n(x^i)\bigr]^{-1} - \bigl[Q^*(x^i)\bigr]^{-1}\bigr\rVert^2 \\
\leq& \frac{1}{n} \sum_{i=1}^{n} \bigl\lVert \bigl[Q^*(x^i)\bigr]^{-1}\bigr\rVert^2 \bigl\lVert \bigl[\hQ_n(x^i)\bigr]^{-1}\bigr\rVert^2 \norm{\hQ_n(x^i) - Q^*(x^i)}^2 \\
\leq& \biggl(\sup_{\bar{x} \in \X} \bigl\lVert \bigl[Q^*(\bar{x})\bigr]^{-1}\bigr\rVert^2\biggr) \biggl(\sup_{\bar{x} \in \X} \bigl\lVert \bigl[\hQ_n(\bar{x})\bigr]^{-1}\bigr\rVert^2\biggr) \biggl( \frac{1}{n} \sum_{i=1}^{n} \norm{\hQ_n(x^i) - Q^*(x^i)}^2 \biggr)
\end{align*}
Therefore, asymptotic and finite sample guarantees for $\norm{\hQ_n(x) - Q^*(x)}$ and $\frac{1}{n} \sum_{i=1}^{n} \norm{\hQ_n(x^i) - Q^*(x^i)}^2$ are sufficient for verifying Assumptions~\ref{ass:regconsist},~\ref{ass:regconvrate}, and~\ref{ass:reglargedev}. 
Theorem~8.5 of~\citet{fan2008nonlinear} can be used to identify conditions under which these asymptotic guarantees hold for local linear estimators on time series data when the dimension of the covariates~$d_x = 1$.
They also note approaches for estimating~$Q^*$ when $d_x > 1$.
Theorem~2 of~\citet{ruppert1997local} can be used to verify Assumptions~\ref{ass:regconsist} and~\ref{ass:regconvrate} for local polynomial smoothers.
Proposition~2.1 and Theorem~3.1 of~\citet{jin2015adaptive} identify conditions under which Assumptions~\ref{ass:regconsist} and~\ref{ass:regconvrate} hold for a local likelihood estimator.
\citet{van2010semiparametric} consider semiparametric models for both~$f^*$ and~$Q^*$.
Theorems~3.1 and~3.2 therein can be used to verify Assumptions~\ref{ass:regconsist} and~\ref{ass:regconvrate} for the estimates~$\hQ_n$. Section~3 of~\citet{zhou2018new} presents robust estimators of~$Q^*$ when~$f^*$ is linear and notes that these estimators~$\hQ_n$ possess asymptotic and finite sample guarantees in the form of Assumptions~\ref{ass:regconsist},~\ref{ass:regconvrate}, and~\ref{ass:reglargedev}.
Finally, Theorem~3.1 of~\citet{chesneau2020nonparametric} can be used to derive asymptotic guarantees for wavelet estimators of~$Q^*$.

\section{Conclusion}

In this note, we propose generalizations of the ER-SAA and ER-DRO frameworks in~\cite{kannan2020data,kannan2020residuals} that can handle heteroscedastic errors, focusing mainly on ER-SAA for brevity.
We identify sufficient conditions 
under which solutions to these approximations possess asymptotic and finite sample guarantees for a class of two-stage stochastic MIPs with continuous recourse.
Furthermore, we outline conditions under which these assumptions hold for some regression setups, including OLS, Lasso, and kNN regression.

Future work includes verification of the large deviation Assumption~\ref{ass:reglargedev} for the regression estimate~$\hQ_n$ for additional prediction setups, consideration of more general relationships between the random vector~$Y$ and the random covariates~$X$, and investigation of the computational performance of the generalizations of the ER-SAA and ER-DRO problems on a practical application involving heteroscedasticity.

\section*{Acknowledgments}

We thank Prof.\ Erick Delage for encouraging us to investigate extensions of the ER-SAA formulation to the heteroscedastic setting.
This research is supported by the Department of Energy, Office of Science, Office of Advanced Scientific Computing Research, Applied Mathematics program under Contract Number DE-AC02-06CH11357.

{
\footnotesize
\section*{References}
\begingroup
\renewcommand{\section}[2]{}%
\bibliographystyle{abbrvnat}
\setlength{\bibsep}{2.7pt}
\bibliography{main}
\endgroup
}

\end{document}